\documentclass[reqno, 11pt, a4paper]{amsart}

\usepackage{latexsym,ifthen,xspace}
\usepackage{amsmath,amssymb,amsthm}
\usepackage{enumerate, calc}
\usepackage[colorlinks=true, pdfstartview=FitV, linkcolor=blue,
  citecolor=blue, urlcolor=blue, pagebackref=false]{hyperref}
\usepackage[text={33pc,605pt},centering]{geometry}    
\usepackage{graphicx}
\newtheorem{theorem}{Theorem}[section]
\newtheorem{lemma}[theorem]{Lemma}
\newtheorem{prop}[theorem]{Proposition}

\theoremstyle{definition}
\newtheorem{definition}[theorem]{Definition}

\theoremstyle{remark}
\newtheorem{remark}[theorem]{Remark}

\numberwithin{equation}{section}

\DeclareMathAlphabet{\mathsl}{OT1}{cmss}{m}{sl}
\SetMathAlphabet{\mathsl}{bold}{OT1}{cmss}{bx}{sl}

\newcommand{\om}{\ensuremath{\omega}}

\newcommand{\cA}{\ensuremath{\mathcal A}}
\newcommand{\cB}{\ensuremath{\mathcal B}}

\newcommand{\cG}{\ensuremath{\mathcal G}}

\newcommand{\cX}{\ensuremath{\mathcal X}}

\newcommand{\bbE}{\ensuremath{\mathbb E}}

\newcommand{\bbP}{\ensuremath{\mathbb P}} 
 
\newcommand{\bbR}{\ensuremath{\mathbb R}}

\newcommand{\bbZ}{\ensuremath{\mathbb Z}} 
\DeclareMathOperator{\prob}{\mathbb{P}}

\DeclareMathOperator{\supp}{\mathrm{supp}}

\newcommand{\ldef}{\ensuremath{\mathrel{\mathop:}=}}

\def\indicator{{\mathchoice {1\mskip-4mu\mathrm l}%
{1\mskip-4mu\mathrm l}{1\mskip-4.5mu\mathrm l}%
{1\mskip-5mu\mathrm l}}}
\def\TH(#1){\label{#1}}\def\thv(#1){\ref{#1}}
\def\Eq(#1){\label{#1}}\def\eqv(#1){(\ref{#1})}

\begin{document}

\title[Diffusion processes on BBM]{Diffusion processes on branching Brownian motion}


%
\author{Sebastian Andres}
\address{University of Cambridge}
\curraddr{Centre for Mathematical Sciences \\
Wilberforce Road, Cambridge
CB3 0WB}
\email{s.andres@statslab.cam.ac.uk}
\thanks{}

\author{Lisa Hartung}
\address{ Courant Institute of Mathematical Sciences ,
New York University,
}
\curraddr{251 Mercer Street, 10012 New York, NY}
\email{lisa.hartung@nyu.edu}
\thanks{This work has been written while the authors were affiliated to the  Rheinische Friedrich-Wilhems Universit\"at Bonn and were partially  supported by the German Research Foundation in the Collaborative Research Center 1060 \lq\lq The Mathematics of Emergent Effects\rq\rq, Bonn.}



\date{\today}

\dedicatory{}

\begin{abstract}

We construct a class of one-dimensional diffusion processes on the particles of branching Brownian motion that are symmetric with respect to the limits of random martingale measures. 
These measures are associated with the extended extremal process of branching Brownian motion and are supported on a Cantor-like set. 
The processes are obtained via a time-change of a standard one-dimensional reflected Brownian motion on $\bbR_+$ in terms of the associated positive continuous additive functionals.

The processes introduced in this paper may be regarded as an analogue of the Liouville Brownian motion which has been recently constructed in the context of a Gaussian free field.
\end{abstract}

\maketitle

\section{Introduction}
Over the last years diffusion processes in random environment, constructed by a random time-change of a standard Brownian motion in terms of singular measures,  appeared in several situations. One prime example is the so-called FIN-diffusion (for Fontes, Isopi and Newman) introduced in \cite{FIN02} which appears 
for instance as the annealed scaling limit for one-dimensional trap models (see \cite{FIN02,BC05,BC06}) and for the one-dimensional random conductance model 
with heavy-tailed conductances (see \cite[Appendix~A]{Ce11}).
Another example is the Liouville Brownian motion, recently constructed in \cite{GRV13,Be13} as the natural diffusion process in the random geometry associated with two-dimensional Liouville quantum gravity.

In this paper we add one more class of examples to the collection. We  consider a time change given by the right-continuous inverse of the positive continuous additive functional whose Revuz measure is the limit of certain random martingale measures that appear in the description of the extremal process of a branching Brownian motion (BBM for short). As a result we obtain a pure jump diffusion process on a Cantor-like set representing the positions of the BBM particles in the underlying Galton-Watson tree.  

Branching Brownian motion has already been introduced in \cite{Moyal, Skorohod64} in the late 1950s and early 1960s. It is a continuous-time Markov branching
process on a probability space $(\Omega, \mathcal{F}, \prob)$ which is constructed as follows.
We start with a continuous-time Galton-Watson process (see e.g.\ \cite{AN}) with branching mechanism $p_k, k\geq 1$, normalised such that $\sum_{i=1}^\infty p_k=1$,
$\sum_{k=1}^\infty k p_k=2$ and $K=\sum_{k=1}^\infty k(k-1)p_k<\infty$. At any time $t$
 we may label the endpoints of the process $i_1(t),\dots, i_{n(t)}(t)$, where $n(t)$ is the number of 
 branches at time $t$. Observe that by our choice of normalisation we have that 
 $\mathbb{E}  n(t)=e^t$. 
BBM is then constructed by starting a Brownian motion at the 
 origin at time 
 zero, running it until the first time the GW process branches, and then starting independent Brownian motions for
 each branch of the GW process starting at the position of the original BM at the branching time. Each of these runs again 
 until the next branching time of the GW occurs, and so on.

 We denote the positions of the $n(t)$ particles at time $t$ by $x_1(t),\dots, x_{n(t)}(t)$. 
 Note that, of course, the positions of these particles do not reflect the position of the 
 particles ``in the tree". 
 
 \begin{remark}
 \label{remark.1}
 By a slight abuse of notation, we also denote by $x_k(s)$ for $s<t$ the particle position of the ancestor of the particle $i_k(t)$ at time $s$. 
 \end{remark}
 
 
Setting $m(t)\ldef\sqrt 2 t - \frac{3}{2 \sqrt{2}} \log(t)$, Bramson \cite{B_M, B_C},  and Lalley and Selke \cite{LS87} showed that
 \begin{equation}\label{extremal.1.1}
 \lim_{t\uparrow\infty}\mathbb{P}\left(\max_{k\leq n(t)}x_k(t)-m(t)\leq x\right)=\mathbb{E} 
 \left[e^{-CZe^{-\sqrt 2 x}}\right],
\end{equation}
 for some constant $C$, where $Z\ldef\lim_{t\uparrow\infty} Z_t$ is the $\prob$-a.s.\ limit of the derivative martingale
 \begin{equation}\label{Z.0}
 Z_t\ldef \sum_{j\leq n(t)}(\sqrt 2 t-x_j(t))e^{\sqrt 2(x_j(t)-\sqrt 2 t)}, \qquad t\geq 0.
 \end{equation}
For $0<r<t$ a truncated version of the derivative martingale
 \begin{align}
 Z_{r,t}(v) \ldef \sum_{j\leq n(t)} \big( \sqrt{2} t -x_j(t) \big) e^{\sqrt{2}(x_j(t)-\sqrt{2}t)} \indicator_{\{\gamma(x_j(r)) \leq v\}}, \qquad v\in \bbR_+,
\end{align}
has been introduced in \cite{BH14}. Here we denote by $\gamma$ an 
  embedding of the particles $\{1,\dots,n(t)\}$ into $\mathbb{R}_+$, which encodes the positions of the particles in the underlying Galton-Watson tree respecting the genealogical distance (see Section~\ref{sec:def_gamma} below for the precise definition). In a sense, the embedding $\gamma$ is a natural continuous-time analogue of the well-established encoding of binary branching processes in discrete time, where the leaves of tree are identified with binary numbers.
 The random measure on $\bbR_+$ associated  with $Z_{r,t}$ is given by
\begin{align}
 M_{r,t} \ldef  \sum_{j\leq n(t)} \big( \sqrt{2} t -x_j(t) \big) e^{\sqrt{2}(x_j(t)-\sqrt{2}t)} \delta_{\gamma(x_j(r))}.
\end{align}
In  \cite{BH14} it has been shown that the vague limit
\begin{align}
 M = \lim_{r\uparrow \infty} \lim_{t\uparrow \infty} M_{r,t} \qquad \text{exists  $\prob$-a.s.}
\end{align}
Furthermore, in \cite[Theorem 3.1] {BH14} an extended convergence result of the extremal process has been proven, namely 
 \begin{equation}\label{extremal.5}
 \sum_{k=1}^{n(t)} \delta_{(\gamma(x_k(t), x_k(t)-m(t))}\Rightarrow  \sum_{i,j}\delta_{(q_i,p_i)+(0,\Delta^{(i)}_j)}, \qquad \text{ on $\mathbb{R}_+\times \mathbb{R}$, as
 $t\uparrow \infty$,}
 \end{equation}
 where  $(q_i,p_i)_{i\in \mathbb{N}}$ are the atoms of a Cox process on $\mathbb{R}_+\times \mathbb{R}$ with intensity measure
 $M( dv)\times Ce^{-\sqrt{2} x}dx$ and $(\Delta_j^{(i)})_{i,j}$  are the atoms of  independent and identically distributed point processes $\Delta^{(i)}$  with 
 \begin{equation}\label{extremal.3}
 \Delta^{(1)} \stackrel{D}{=} \lim_{t\uparrow \infty}
 \sum_{i=1}^{n(t)}\delta_{\tilde x_i(t)-\max_{j\leq n(t)}\tilde x_j(t)},
 \end{equation}
 where $\tilde x(t)$ is a BBM conditioned on $\max_{j\leq n(t)} \tilde x_j(t)\geq \sqrt 2 t$.
Recall that in \cite{ABK_E, ABBS} it was already shown that
$ \sum_{k=1}^{n(t)}\delta_{x_k(t)-m(t)}$ converges to the Poisson cluster process given by the projection of the limit in \eqref{extremal.5} onto the second coordinate.

\subsection{Results}
Let $(\Omega', (B_s)_{s\geq 0}, \mathcal{G}, (\mathcal{G}_s)_{s\geq 0}, (P_x)_{x\in \bbR_+})$ denote a one-dimensional reflected standard Brownian motion $B$ on $\bbR_+$. Recall that $B$ is reversible w.r.t.\ the Lebesgue measure $dx$ on $\bbR_+$. 
Then, the positive continuous additive functional (PCAF) of $B$ having Revuz measure $M_{r,t}$ (see Appendix~\ref{app:PCAF} for definitions) is given by $F_{r,t}: [0,\infty) \rightarrow [0,\infty)$ defined as
\begin{align}\label{eq:def_Frt}
 F_{r,t}(s) \ldef  \int_{\bbR_+} L_s^a \, M_{r,t}(da) =  \sum_{j=1}^{n(t)} \big( \sqrt{2} t -x_j(t) \big) e^{\sqrt{2}(x_j(t)-\sqrt{2}t)} L^{\gamma(x_j(r))}_s,
\end{align}
where $\{L^a, a\in \bbR\}$ denotes the family of local times of $B$. 
Further, we define
 \begin{align} \label{eq:def_F}
  F(s) \ldef \int_{\bbR_+} L_s^a \, M(da), \qquad s\geq 0.
 \end{align}

\begin{theorem} \label{thm:constrF}
 $\prob$-a.s., the following hold. 
\begin{enumerate}
 \item[(i)] There exist a set $\Lambda \subset \Omega'$ with $P_x[\Lambda]=1$ for all $x\in \bbR_+$ on which 
\begin{align} 
 F=\lim_{r\uparrow \infty} \lim_{t\uparrow\infty} F_{r,t}, \qquad \text{in $\sup$-norm on $[0,S]$,}
 \end{align}
 for any $S>0$. In particular, $F$ is continuous, 	increasing and satisfies $F(0)=0$ and $\lim_{s\to \infty} F(s)=\infty$.
 \item[(ii)]  The functional $F$ is  the (up to equivalence) unique PCAF of $B$ with Revuz measure $M$. 
\end{enumerate}
\end{theorem}

\begin{definition}
We define  the  process $\cB$  as the time-changed Brownian motion 
\begin{align}
 \cB(s) \ldef B_{F^{-1}(s)}, \qquad s\geq 0,
\end{align} 
where $F^{-1}$ denotes the right-continuous inverse of the PCAF $F$ in \eqref{eq:def_F}.
\end{definition}

By the general theory of time changes of Markov processes, in particular cf.\ \cite[Theorem~6.2.1]{FOT11}, $\cB$ is a right-continuous strong Markov process on $\supp M$, which is $M$-symmetric and induces a strongly continuous transition semigroup. 
Note that the empty set is the only polar set for the one-dimensional Brownian motion, so the measure $M$ does trivially not charge polar sets. Further, for any $0<r<t$ set
\begin{align}
  \cB_{r,t}(s) \ldef B_{F_{r,t}^{-1}(s)}, \qquad s\geq 0,
\end{align}
where $F_{r,t}^{-1}$ denotes the right-continuous inverse of $F_{r,t}$. Then, as $r$ and $t$ tend to infinity, the processes $\cB_{r,t}$ converge  in law  towards $\cB$ on the Skorohod space $D((0,\infty), \bbR_+)$ equipped with $L^1_{\mathrm{loc}}$-topology (see Theorem~\ref{thm:approx_proc} below). In a sense $\cB_{r,t}$ may be regarded as a random walk on the leaves of the underlying Galton-Watson tree. In addition, we also provide an approximation result for $\cB$ in terms of random walks on a lattice (see Theorem~\ref{thm:rw_approx} below).


Similarly to the above procedure, for any $\sigma \in (0,1)$, one obtains a measure $M^\sigma$ from a truncation of the McKean martingale  
\begin{align}
Y^{\sigma}_t\ldef \sum_{i=1}^{n(t)} e^{\sqrt{2}\sigma x_k(t)-(1+\sigma^2)t}, \qquad t\geq 0.
\end{align}
Then one can define the process $\cB^\sigma$ as $\cB^\sigma(s)\ldef B_{(F^\sigma)^{-1}(s)}$ with  $F^\sigma$ being the PCAF  associated with $M^\sigma$. We refer to Section~\ref{sec:subcrit} for further details.

\bigskip
A diffusion process being similar to but different from $\cB$ is the FIN-diffusion introduced in \cite{FIN02}. It is a one-dimensional singular diffusion in random environment given by a random speed measure $\rho=\sum_i v_i \delta_{y_i}$, where $(y_i,v_i)$ is  an inhomogeneous Poisson
point process on $\bbR \times (0,\infty)$ with intensity measure $dy \, \alpha v^{-1-\alpha} \, dv$ for $\alpha \in (0,1)$. Let $F_{\mathrm{FIN}}$ be the PCAF 
\begin{align}
F_{\mathrm{FIN}}(s) \ldef \int_{\bbR} L^a_s(W) \, \rho(da)
\end{align}
with $\{L^a(W), a\in \bbR\}$ denoting the family of local times of a one-dimensional Brownian motion $W$. 
Then, the FIN-diffusion $\{ \mathrm{FIN}(s), s\geq 0\}$ is the diffusion process defined as the time change $\mathrm{FIN}(s) \ldef W_{(F_{\mathrm{FIN}})^{-1}(s)}$ of the Brownian motion $W$.
At first sight the measure $\rho$ and the process $\mathrm{FIN}$ resemble strongly $M$ and $\cB$, respectively. However, one significant difference is that $\rho$ is a discrete random measure with a set of atoms being dense in $\bbR$, so that $\rho$ has full support $\bbR$ and  $\mathrm{FIN}$ has continous sample paths (see \cite{FIN02} or \cite[Proposition~3.2]{BC06}), while the measure $M$ is concentrated on a Cantor-like set and the sample paths of $\cB$ have jumps.

\bigskip
Another prominent example for a log-correlated process is the Gaussian Free Field (GFF) on a two-dimensional domain. In a sense the processes $\cB$ or $\cB^\sigma$ introduced in this paper can be regarded as the BBM-analogue of the Liouville Brownian motion (LBM) recently constructed in \cite{GRV13} and in a weaker form in \cite{Be13}. More precisely, let $X$ denote a (massive) GFF on a domain $D\subseteq \bbR^2$, then in the subcritical case the analogue of the martingale measure $M^\sigma$ 
can be constructed by using the theory of Gaussian multiplicative chaos established by Kahane in \cite{Ka85} (see also \cite{RV14} for a review).
On a formal level the resulting so-called Liouville measure on $D$ is given by
\begin{align}
e^{\gamma X(z)-\frac{\gamma^2}{2} \bbE[X(z)^2]} \, dz, \qquad \gamma\in (0,2).
\end{align}
The associated PCAF $F_\mathrm{LBM}$, which can formally be written as
\begin{align}
 F_{\mathrm{LBM}}(s)=\int_0^s e^{\gamma X(W_r)-\frac{\gamma^2}{2} \bbE[X(W_r)^2]} \, dr,
\end{align}
where $W$ denotes a two-dimensional standard Brownian motion on the domain $D$, has been constructed in \cite{GRV13} (cf.\ also \cite[Appendix~A]{AK14}). Then, the Liouville Brownian motion $\{\mathrm{LBM}(s), s\geq 0\}$ is defined as $\mathrm{LBM}(s)\ldef W_{F_\mathrm{LBM}^{-1}(s)}$.

In the critical case $\gamma=2$ the corresponding analogue of the derivative martingale measure $M$ can be interpreted as being given by
\begin{align}
 -  \big(X(z)-2\bbE[X(z)^2]\big) \, e^{2 (X(z)-\bbE[X(z)^2])} \, dz,
\end{align}
which has been introduced in \cite{DRSV14,DRSV14a}. The corresponding PCAF and the critical Liouville Brownian motion have been constructed in \cite{RV15}.
In the context of a discrete GFF such measures have been studied in \cite{BisLou13, BisLou14, BisLou16}, where in \cite{BisLou13} an analogue of the extended convergence result in \eqref{extremal.5} has been established.

However, a major difference between the processes $\cB$ and $\mathrm{LBM}$ is that for the LBM the functional $F_\mathrm{LBM}$ and the planar Brownian motion $W$ are independent (cf.\ \cite[Theorem 2.21]{GRV13}), while in the present paper the functional $F$ and the Brownian motion $B$ are \emph{dependent}  since $L$ is the local time of $B$. A similar phenomenon can be observed in the context of trap models, where in dimension $d=1$ the underlying Brownian motion and the clock process of the FIN diffusion are dependent and in dimension $d \geq 2$ the Brownian motion and the clock process of the scaling limit, known as the so-called fractional kinetics motion,  are independent.

\bigskip

In \cite{CHK16} Croydon, Hambly and Kumagai consider time-changes of stochastic processes and their discrete approximations in a quite general framework for the case when the underlying process is point recurrent, meaning that it can be described in terms of its resistance form (examples include the one-dimensional standard Brownian motion or Brownian motion on tree-like spaces and certain low-dimensional fractals). The results cover the FIN-diffusion and a one-dimensional version of the LBM. However, the results of the present paper do not immediately follow from the approximation result in \cite{CHK16} since the required convergence of the measures $M_{r,t}$ towards $M$ in the Gromov-Hausdorff-vague topology on the non-compact space $\bbR_+$ needs to be verified.  

\bigskip

The rest of the paper is organised as follows. In Section~\ref{sec:prelim}  we first recall the definitions of a PCAF and its Revuz measure and we provide the precise definition of the embedding $\gamma$ and the (truncated) critical martingale measures. Then we prove Theorem~\ref{thm:constrF} in Section~\ref{sec:apprPCAF} and we specify some properties of the process $\cB$, in particular we describe its Dirichlet form. In Section~\ref{sec:rw} we show random walk approximations of $\cB$. In Section~\ref{sec:subcrit} we sketch the construction of the process $\cB^\sigma$ associated with the martingale measure obtained from the McKean martingale. Finally, in the appendix we collect some properties of Brownian local times needed in the proofs.


\section{Preliminaries} \label{sec:prelim}

 \subsection{Additive functionals and Revuz measures} \label{app:PCAF}
First we briefly recall the definition of an additive functional of a symmetric Markov process and some of its main properties, for more details on this topic see e.g.\ \cite{FOT11, CF12}. Let
 $E$ be a locally compact separable metric space and let $m$ be a positive Radon measure on $E$ with $\supp(m)=E$. We consider an $m$-symmetric conservative Markov process
$(\Omega',\cG, (\cG_t)_{t\geq 0}, (X_t)_{t\geq 0}, (P_x)_{x\in E})$ and denote by $\{\theta_t\}_{t\geq 0}$ be the family of shift mappings on
$\Omega'$, i.e.\ $X_{t+s}=X_t \circ \theta_s$ for $s,t \geq 0$. 

\begin{definition}
i) A $[0,\infty]$-valued stochastic process $A=(A_t)_{t\geq 0}$
on $(\Omega',\mathcal{G})$
is called a \emph{positive continuous additive functional (PCAF)} of $X$
(in the strict sense), if $A_t$ is $\mathcal{G}_t$-measurable for every
$t\geq 0$ and if there exists a set $\Lambda \in \mathcal{G}$,
called a \emph{defining set} for $A$, such that
\begin{enumerate}
\item [a)] for all $x\in E$, $P_x[\Lambda]=1$,
\item [b)] for all $t\geq 0$, $\theta_t(\Lambda)\subset \Lambda$,
\item [c)] for all $\om \in \Lambda$,
$[0,\infty)\ni t\mapsto A_t(\om)$ is a $[0,\infty)$-valued continuous function
with $A_0(\om)=0$ and
\begin{align}
 A_{t+s}(\om)=A_t(\om)+ A_s \circ \theta_t (\om), \qquad \forall s,t\geq 0.
\end{align}
\end{enumerate}

ii) Two such functionals $A^1$ and $A^2$ are called \emph{equivalent}
if $P_x[A^1_t=A^2_t]=1$ for all $t>0$, $x\in E$, or equivalently,
there exists a defining set $\Lambda \in \mathcal{G}_\infty$ for
both $A^1$ and $A^2$ such that $A^1_t(\om)=A^2_t(\om)$ for all $t\geq 0$,
$\om\in \Lambda$.

iii) For any such $A$, a Borel measure $\mu_A$ on $E$ satisfying 
\begin{align} \label{eq:rev_corr_gen}
 \int_{E} f(y) \, \mu_A(dy)
	=\lim_{t\downarrow 0} \frac 1 t \int_{E} E_x \Bigl[ \int_0^t f(B_s) \, dA_s \Bigr] \, m(dx)
\end{align}
for any non-negative Borel function $f:E\to[0,\infty]$ is called the
\emph{Revuz measure} of $A$, which exists uniquely by general theory
(see e.g.\ \cite[Theorem~A.3.5]{CF12}). 
\end{definition}

We recall that for a given a Borel measure $\mu_A$ charging no polar sets a PCAF $A$ satisfying \eqref{eq:rev_corr_gen} exists uniquely up to equivalence (see e.g.\ \cite[Theorem 5.1.3]{FOT11}).
Observe that in the present setting where $m$ is invariant the measure $\mu_A$ is already characterised by the simpler formula
\begin{align} \label{eq:revuz_invariant}
 \int_{E} f(y) \, \mu_A(dy)
	=\int_{E} E_x \Bigl[ \int_0^1 f(B_s) \, dA_s \Bigr] \, m(dx).
\end{align}

\subsection{Definition of the embedding} \label{sec:def_gamma}

We start by recalling the definition of the embedding $\gamma$ given in \cite{BH14} which is a slight variant of the familiar Ulam-Neveu-Harris labelling 
(see e.g.\ \cite{HaHa06}). 
We denote the set of (infinite) multi-indices by
$\mathbf{I}\equiv \mathbb{Z}_+^{\mathbb{N}},$
and let $\mathbf{F}\subset \mathbf{I}$ be the subset of multi-indices that contain 
only finitely many entries different from zero. Ignoring leading zeros, we see
that 
\begin{equation}\label{multi.2}
\mathbf{F} = \cup_{k=0}^\infty \mathbb{Z}_+^k,
\end{equation}
where $\mathbb{Z}_+^0$ is either the empty multi-index or the multi-index containing only 
zeros.

We encode a continuous-time Galton-Watson process by the set of branching times,
$\{t_1<t_2<\dots< t_{W(t)}<\dots\}$, where $W(t)$ denotes the number of branching 
times up to time $t$, and by a consistently assigned set of multi-indices
for all times $t\geq 0$. 
To do so, (for a given tree) 
the sets of multi-indices, $\tau(t)$ at time $t$, are constructed as follows.

\begin{figure}[htbp]
 \includegraphics[width=13.5cm]{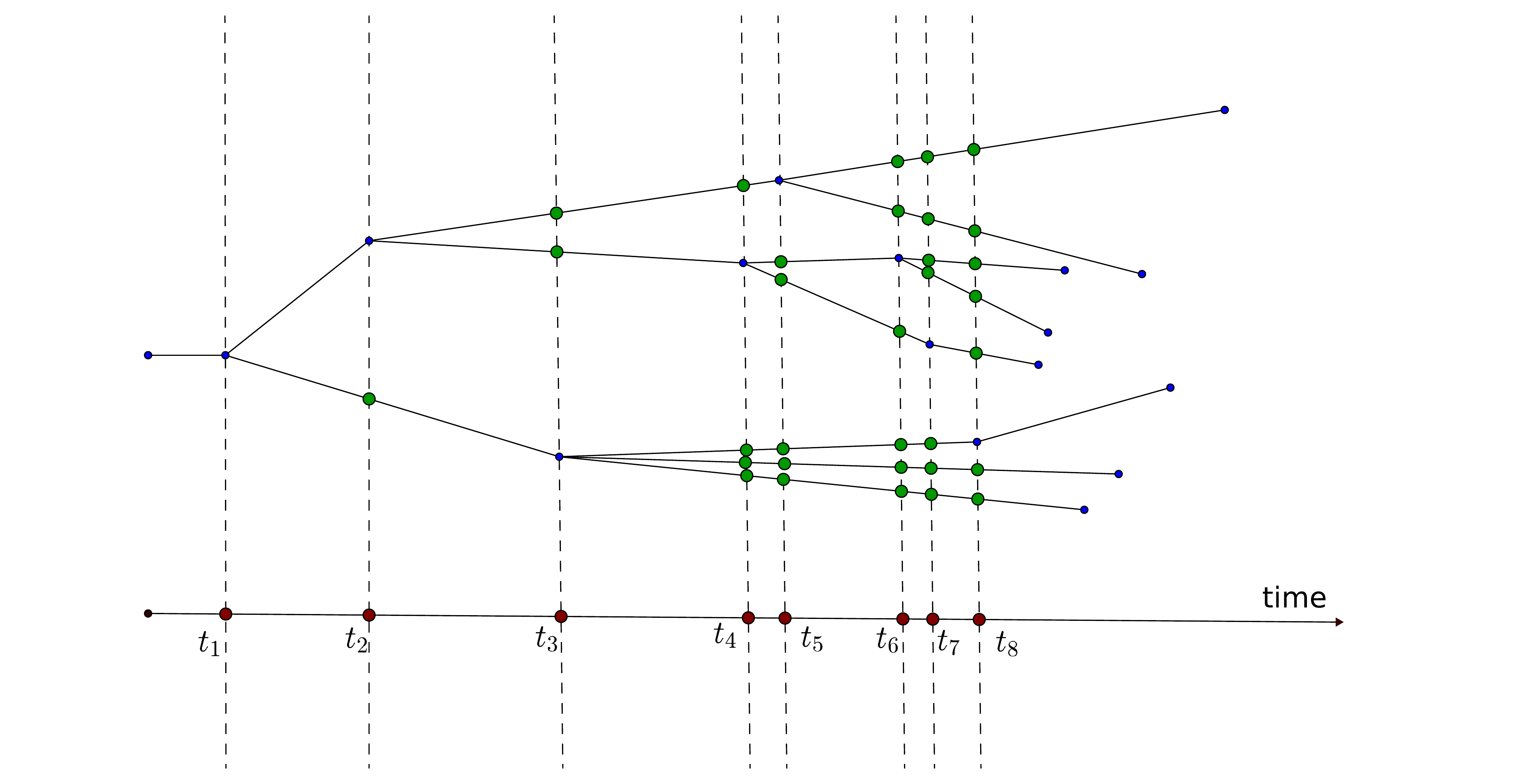}
 \caption{Construction of $\widetilde T$: The green nodes were introduced into the tree `by hand'.}\label{figure.1}
\end{figure}

\begin{itemize}
 \item $\{(0,0,\dots)\} =\{u(0)\}=\tau(0)$.
 \item  for all $j\geq 0$,  for all  $t\in [t_j,t_{j+1})$,   $\tau(t)=\tau(t_j)$.
 \item If $u\in \tau(t_j)$ then $u+(\underbrace{0,\dots,0}_{W(t_j) \times 0}, k,0,\dots)\in 
 \tau(t_{j+1})$ if $0\leq k\leq l^u(t_{j+1})-1$, where
 \begin{equation}\label{map.1}
l^u(t_j)=\# \{\mbox{ offsprings of  the particle corresponding to }u\, \mbox{at time}\, t_j \}.
\end{equation}
\end{itemize}
We use the convention that, if a given branch of the tree does not ``branch" at time $t_j$, we
add to the underlying 
Galton-Watson  at this time an extra vertex where $l^u(t_j)=1$ (see Figure \ref{figure.1}). We call the
resulting tree $\widetilde  T_t$. 

One relates the assignment of labels in the following backward consistent way. 
For  $u\equiv (u_1,u_2,u_3,\dots)\in \mathbb{Z}_+^\mathbb{N}$, we define the function $u(r), r\in \mathbb{R}_+$,  through
 \begin{equation}\label{ multi.3}
u_\ell(r)\equiv \begin{cases}  u_\ell,&\,\, \mbox{  if}\,\, t_\ell\leq r,\\
0,&\,\, \mbox{  if}\,\, t_\ell> r.
\end{cases}
\end{equation}
Clearly,  if $u(t)\in \tau(t)$ and  $r\leq t$, then $u(r)\in \tau(r)$.  This allows to define the 
\emph{boundary} of the 
tree at infinity by
$\partial \mathbf{T} \equiv \left\{ u\in \mathbf{I}: \forall t<\infty, u(t)\in \tau(t)\right\}$.
\index{boundary!of Galton-Watson process}
 In this way we identify each leaf of the Galton-Watson tree at time $t$,  $i_{k}(t)$ with $k\in\{1,\dots,n(t)\}$, with some multi-label $u^k(t)\in \tau(t)$. We define the embedding $\gamma$ by
\begin{equation}\label{map.3}
\gamma(u(t))\equiv\sum_{j=1}^{W(t)} {u_{j}(t)}e^{-t_j}.
\end{equation}
For a given $u$, the function $(\gamma(u(t)), t\in \mathbb{R}_+)$ describes a trajectory of a particle in $\mathbb{R}_+$,
which converges to some point $\gamma(u)\in 
\mathbb{R}_+$, as $t\uparrow \infty$, $\prob$-a.s. Hence also the sets $\gamma(\tau(t))$ converge, for any realisation of the tree, to some (random) set $\gamma(\tau(\infty))$.  


Recall that in BBM there is also the position of the Brownian motion
$x_k(t)$ of the $k$-th particle at time $t$. Thus to any ``particle" at time $t$ we can now associate the position $( \gamma(u^k(t)),x_k(t))$, in $\mathbb{R}_+\times\mathbb{R}$. Hoping that there will not be too much confusion, 
we will identify $\gamma(u^k(t))$ with $\gamma(x_k(t))$. 

\subsection{The critical martingale measure}

A key object is the derivative martingale $Z_t$  defined in $\eqref{Z.0}$. Recall the following result proven in \cite{LS87}.
\begin{lemma} \label{lem:Z}
The limit  $Z\ldef \lim_{t\to \infty} Z_t$ exists $\prob$-a.s.\ and $\min_{i\leq n(t)} (\sqrt 2 t - x_i(t)) \rightarrow \infty$ as $t\to \infty$ $\prob$-a.s.
\end{lemma}
For $0<r<t$ the truncated version
\begin{align}
 Z_{r,t}(v) \ldef \sum_{j\leq n(t)} \big( \sqrt{2} t -x_j(t) \big) e^{\sqrt{2}(x_j(t)-\sqrt{2}t)} \indicator_{\{\gamma(x_j(r)) \leq v\}}, \qquad v\in \bbR_+,
\end{align}
has been recently introduced in \cite{BH14}. In particular, by \cite[Lemma 3.2]{BH14} for every $v\in \bbR_+$ the limit
\begin{align} \label{eq:conv_Zrt}
 Z(v)\ldef \lim_{r\uparrow \infty} \lim_{t\uparrow \infty} Z_{r,t}(v) 
\end{align}
exists $\prob$-a.s. Consider now the associated measures on $\bbR_+$ given by
\begin{align}
 M_{r,t} \ldef  \sum_{j\leq n(t)} \big( \sqrt{2} t -x_j(t) \big) e^{\sqrt{2}(x_j(t)-\sqrt{2}t)} \delta_{\gamma(x_j(r))},
\end{align}
and denote by $M$ the Borel measure on $\bbR_+$ defined via $M([0,v])=Z(v)$ for all $v\in \bbR_+$. Then, \eqref{eq:conv_Zrt} implies that $\prob$-a.s.
\begin{align}
 M = \lim_{r\uparrow \infty} \lim_{t\uparrow \infty} M_{r,t} \qquad \text{vaguely}.
\end{align}
By \cite[Proposition~3.2]{BH14}  $M$ is $\prob$-a.s.\ non-atomic. Moreover, due to the recursive structure of the underlying GW-tree
 $M$ is supported on some Cantor-like set $\cX$.

\section{Approximation of the PCAF and properties of $\cB$} \label{sec:apprPCAF}
\subsection{Proof of Theorem~\ref{thm:constrF}}

Let $\Omega':=C([0,\infty),\bbR)$ and let $W=(W_t)_{t\geq 0}$ be the coordinate
process on $\Omega'$ and set $\mathcal{G}^0_\infty:=\sigma(W_s; \, s<\infty)$
and $\mathcal{G}^0_t:=\sigma(W_s; \, s\leq t)$, $t\geq 0$. Further, let
$\{P_x\}_{x\in \bbR}$ be the family of probability measures on
$(\Omega',\mathcal{G}^0_\infty)$ such that for each $x\in \bbR$,
$W=(W_t)_{t\geq 0}$ under $P_x$ is a one-dimensional Brownian motion starting
at $x$. We denote by $\{ \mathcal{G}_t\}_{t\in [0,\infty]}$ the minimum completed
admissible filtration for $W$ and by $L(W)=\{L_t^a(W), t\geq 0, a\in \bbR\}$  the random field of local times of $W$.

Now we set $B_t\ldef |W_t|$, $t\geq 0$, so that $(\Omega',\cG, (\cG_t)_{t\geq 0}, (B_t)_{t\geq 0}, (P_x)_{x\in \bbR_+})$ is a reflected Brownian motion on $\bbR_+$. Then, the family $L\equiv L(B)=\{L_t^a(B), t\geq 0, a\in \bbR_+\}$ of local times of $B$ is given by
\begin{align} \label{eq:def_L}
 L_t^a\equiv L_t^a(B)=L_t^a(W)+ L_t^{-a}(W), \qquad t\geq 0, \, a\in \bbR_+
\end{align}
(cf.\ \cite[Exercise VI.1.17]{RY99}). 

\begin{prop} \label{prop:FrtPCAF}
 For $\bbP$-a.e.\ $\om$, there exists $\tau_0=\tau_0(\om)$ such that for all $t\geq \tau_0$ and $0\leq r <t$ the following hold.
\begin{enumerate}
 \item[(i)] The unique PCAF of $B$ with Revuz measure $M_{r,t}$ is given by
\begin{align}\Eq(lisa.F1)
 F_{r,t}: [0,\infty)\rightarrow [0,\infty) \quad  s \mapsto    \sum_{j=1}^{n(t)} \big( \sqrt{2} t -x_j(t) \big) e^{\sqrt{2}(x_j(t)-\sqrt{2}t)} L^{\gamma(x_j(r))}_s.
\end{align}

\item[(ii)] There exist a set $\Lambda \subset \Omega'$ with $P_x[\Lambda]=1$ for all $x\in \bbR_+$, on which $F_{r,t}$ is continuous, increasing and
satisfies $F_{r,t}(0)=0$ and $\lim_{s \to \infty} F_{r,t}(s)=\infty$.
\end{enumerate}
\end{prop}
\begin{proof}
 Recall that $\min_{i\leq n(t)} (\sqrt 2 t - x_i(t)) \rightarrow \infty$ $\prob$-a.s.\ as $t\to \infty$ by Lemma~\ref{lem:Z}. Then, the statement follows immediately from Lemma~\ref{lem:loc_dirac} and Lemma~\ref{lem:propL}. 
\end{proof}

We now turn to the proof of Theorem~\ref{thm:constrF}.

\begin{proof}[Proof of Theorem~\ref{thm:constrF} (i)]
 Fix any environment $\om \in \Omega$ such that Proposition~\ref{prop:FrtPCAF} holds and $(M_{r,t})$ converges vaguely to $M$ on $\bbR_+$.  In particular,
 \begin{align} \label{eq:vague_conv}
 \lim_{r\uparrow \infty}  \lim_{t\uparrow \infty} \int_{\bbR_+} f(a) \, M_{r,t}(da) =  \int_{\bbR_+} f(a) \, M(da)
 \end{align}
 for all continuous functions $f$ on $\bbR_+$ with compact support. 
 
 By Lemma~\ref{lem:propL} there exists a set $\Lambda \subset \Omega'$ with $P_x[\Lambda]=1$ for all $x\in \bbR_+$ such that $(a,t)\mapsto L_t^a(\om')$ is jointly continuous for all $\om'\in \Lambda$. In particular, for any fixed $s\in [0,S]$ we have that $a\mapsto L_s^a(\om')$ is continuous with compact support $\big[0, \sup_{r\leq s} B_r(\om')\big]$.
Now, by choosing $f(a)=L_s^a(\om')$ in \eqref{eq:vague_conv} we obtain
\begin{align}
\lim_{r\uparrow \infty}  \lim_{t\uparrow \infty} \int_{\bbR_+} L_s^a(\om') \, M_{r,t}(da) =  \int_{\bbR_+} L_s^a(\om') \, M(da),
\end{align}
and therefore pointwise convergence of $F_{r,t}$  towards $F$ on $[0,S]$. 
Recall that by Proposition~\ref{prop:FrtPCAF} the functionals  $F_{r,t}$ are increasing for $t\geq \tau_0(\om)$. Since pointwise convergence of continuous increasing functions towards a continuous function on a compact set implies uniform convergence, the claim follows.
\end{proof} 
 
\begin{remark}
Alternatively, Theorem~1.2 (i) can also be derived from the result in \cite[Theorem~1 (3)]{St63}.  
\end{remark}

For the identification of $F$  as the unique PCAF with Revuz measure $M$ we need a preparatory lemma.

\begin{lemma} \label{lem:uiFrt}
 For $\bbP$-a.e.\ $\om$, there exists $r_0=r_0(\om)$ such that the following holds. For any $x\in \bbR_+$, $S>0$ and any bounded Borel measurable function $f: \bbR_+ \rightarrow [0,\infty)$ the family $\{ \int_0^S f(B_s) \, dF_{r,t}(s) \}_{t\geq r \geq r_0}$ is uniformly $P_x$-integrable.
\end{lemma}
\begin{proof}
 Recall that $\bbP$-a.s.\  $Z_t \rightarrow Z$ (cf.\ Lemma~\ref{lem:Z}), so for $\bbP$-a.e.\ $\om$ there exists $r_0=r_0(\om)$ such that $Z_t\leq 2Z$ for all $t\geq r_0$. 
 It suffices to prove that $\prob$-a.s.\ for any $x\in \bbR_+$,
 \begin{align} \label{eq:uiFrt}
  \sup_{t\geq r \geq r_0} E_x \left[  \Big| \int_0^S f(B_s) \, dF_{r,t}(s) \Big| \right] <\infty.
 \end{align}
 Note that
 \begin{align}
  \int_0^S f(B_s) \, dF_{r,t}(s)  = \sum_{j\leq n(t)} \big( \sqrt{2} t -x_j(t) \big) e^{\sqrt{2}(x_j(t)-\sqrt{2}t)} f(\gamma(x_j(r))) L_S^{\gamma(x_j(r))},
 \end{align}
so that
\begin{align}
  E_x \left[  \Big| \int_0^S f(B_s) \, dF_{r,t}(s) \Big| \right] \leq \|f \|_\infty \, |Z_t| \,  E_x\Big[\sup_{a\in \bbR_+}L_S^a\Big] \leq 2 \|f \|_\infty \, Z \, E_x\Big[ \sup_{a\in \bbR_+} L_S^a\Big],
\end{align}
and \eqref{eq:uiFrt} follows from Lemma~\ref{lem:tail_L}.
\end{proof}

\begin{proof}[Proof of Theorem~\ref{thm:constrF} (ii)] Recall that only the empty set is polar for $B$. In particular, the measure $M$ does trivially not charge polar sets, so by general theory (see e.g.\ \cite[Theorem 4.1.1]{CF12}) the PCAF with Revuz measure $M$ is (up to equivalence) unique. Thus, we need show that the limiting  functional $F$ is $\prob$-a.s.\ in Revuz correspondence with $M$.  In view of \eqref{eq:revuz_invariant} it suffices to prove that $\bbP$-a.s.
 \begin{align} \label{eq:identF}
  \int_{\bbR_+} f(a) \, M(da)
	=\int_{\bbR_+} E_x \Bigl[ \int_0^1 f(B_s) \, dF(s) \Bigr] \, dx
 \end{align}
 for any non-negative Borel function $f: \bbR_+ \rightarrow [0,\infty]$. By a monotone class argument it is enough to consider continuous functions $f$ with compact support in $\bbR_+$.
Note that $E_x[\int_0^1 f(B_s) \, dL^a_s]=f(a) E_x[L_1^a]$ for any $a\in \bbR_+$ and therefore
\begin{align} \label{eq:pre_limit}
 E_x \Big[ \int_0^1 f(B_s) \, dF_{r,t}(s) \Big]=\int_{\bbR_+} f(a) E_x[L_1^a] \, M_{r,t}(da).
\end{align}
By Lemma~\ref{lem:tail_L} we have $ \sup_{a\in \bbR_+} E_x[L_1^a]<\infty$  and together with Lemma~\ref{lem:propL} this implies that the mapping $a\mapsto f(a) E_x[L_1^a]$ is bounded and continuous on $\bbR_+$. Furthermore, by (i) $\prob$-a.s.\ the sequence $(dF_{r,t})$ converges weakly to $dF$ on $[0,1]$, $P_x$-a.s.\ for any $x\in\bbR_+$. 
We take limits in $t$ and $r$ on both sides of \eqref{eq:pre_limit}, where we use Lemma~\ref{lem:uiFrt} for the left hand side and the vague convergence of $M_{r,t}$ towards $M$ for the right hand side, and obtain
\begin{align}
 E_x \Big[ \int_0^1 f(B_s) \, dF(s) \Big]=\int_{\bbR_+} f(a) E_x[L_1^a] \, M(da).
\end{align}
Finally, by integrating both sides over $x\in \bbR_+$ and using Fubini's theorem and Lemma~\ref{lem:int_meanL} we get \eqref{eq:identF}.
\end{proof}


\subsection{First properties of $\cB$} \label{sec:properties} 

Recall that the  process $\cB$  is defined as the time-changed Brownian motion 
\begin{align}
 \cB(s) \ldef B_{F^{-1}(s)}, \qquad s\geq 0,
\end{align} 
where $F$ is the PCAF in \eqref{eq:def_F}.
First, we observe that the continuity of $F$ ensures that the process $\cB$ does not get stuck anywhere in the state space, and $\mathcal{B}$ does not explode in finite time since, $\prob \times P_x$-a.s., $\lim_{s\to \infty} F(s)=\infty$.
However,  $F$ is not strictly increasing so that jumps occur. 

More precisely, by the general theory of time
changes of Markov processes we have the following properties of $\cB$. 
First, in view of  \cite[Theorems~A.2.12]{FOT11} $\cB$ is a right-continuous strong Markov process on $\cX\ldef \supp M$ and by \cite[Proposition~A.3.8]{CF12}  we have $\prob$-a.s.\
\begin{align}
 P_x\big[\cB(s)\in \tilde\cX, \,  \forall s\geq 0 \big]=1, \qquad \forall x\in \cX,
\end{align}
where $\tilde\cX$ denotes the support of the PCAF $F$, i.e.\
\begin{align}
 \tilde \cX \ldef \big\{ x\in \bbR_+: P_x[R=0]=1 \big\} \quad \text{with} \quad R\ldef \inf\{ s>0: F_s>0\}.
\end{align}
By general theory (cf.\ \cite[Section 5.1]{FOT11}) we have $\tilde \cX \subseteq \cX$ (recall that only the empty set is polar) and $\cX \setminus \tilde \cX$ has $M$-measure zero. 

Furthermore,by \cite[Theorem~6.2.3]{FOT11} the process $\cB$ is recurrent and by \cite[Theorem~6.2.1 (i)]{FOT11} the transition function $(P_s)_{s>0}$  of $\cB$ given by
\begin{align}
P_s f(x):=E_x[f(\mathcal{B}(s))], \qquad s>0, \, x\in \cX, \, f\in L^2(\cX,M), 
\end{align}
 determines a strongly continuous semigroup and is $M$-symmetric, i.e.\ it satisfies
\begin{align}
 \int_{\cX} P_s f \cdot g \, dM = \int_{\cX} f \cdot P_s g \, dM 
\end{align}
for all Borel measurable functions $f,g:\ \cX \rightarrow [0,\infty]$.

\subsection{The Dirichlet form}
We can apply the general theory of Dirichlet forms to obtain a more precise
description of the Dirichlet form associated with $\mathcal{B}$.
For $D=(0,\infty)$ denote by $H^1(D)$ the standard Sobolev space, that is
\begin{align}
 H^1(D)=\big\{ f\in L^2(D, dx):\, f' \in L^2(D, dx) \big\},
\end{align}
where the derivatives are in the distributional sense.
On  $H^1(D)$ we define the form
\begin{align} \label{eq:def_dirform}
 \mathcal{E}(f,g)=\frac 1 2 \int_{\bbR_+}  f' \cdot  g' \, dx.
\end{align}
Recall that $(\mathcal{E}, H^1(D))$ can be regarded as a regular Dirichlet form on $L^2(\bbR_+)$ and the associated process is the reflected
Brownian motion $B$ on $\bbR_+$. By $H^1_e(\bbR_+)$ we denote the extended Dirichlet space,
that is the set of $dx$-equivalence classes of Borel measurable functions $f$
on $\bbR_+$ such that $\lim_{n\to \infty} f_n=f \in \bbR$ $dx$-a.e.\ for some
$(f_n)_{n\geq 1} \subset H^1(\bbR_+)$ satisfying
$\lim_{k,l \to \infty}\mathcal{E}(f_k-f_l, f_k-f_l)=0$.
By \cite[Theorem~2.2.13]{CF12} we have the following identification of $H^1_e(D)$:
\begin{align}
 H^1_e(D)=\big\{ f \in L^2_{\mathrm{loc}}(D, dx): f' \in L^2(D, dx) \big\}. 
\end{align}

Recall that $\cX$ denotes the support of the random measure $M$. We define the hitting distribution
\begin{align}
 H_{\cX} f(x) \ldef E_x\big[f(B_{\sigma_{\cX}})\big], \qquad x\in \bbR_+,
\end{align}
with $\sigma_{\cX}\ldef \inf\{t>0: B_t \in \cX \}$ for any non-negative Borel function $f$ on $\bbR_+$.
Note that the function $H_{\cX} f$ is uniquely determined by the restriction of $f$ to the
set $\cX$. Further, by \cite[Theorem 3.4.8]{CF12}, we have $H_{\cX}f \in H^1_e(D)$ and by \cite[Lemma 6.2.1]{FOT11} $H_{\cX}f=H_{\cX}g$ whenever $f=g$ $M$-a.e.\  for any $f,g\in H^1_e(D)$.
Therefore it makes sense to define the symmetric form $(\hat{ \mathcal{E}},  \hat{\mathcal{F}})$ on $L^2({\cX},M)$ by
\begin{align}
 \begin{cases}
   \hat{\mathcal{F}}  \ldef \big\{ \varphi \in L^2({\cX},M): \varphi= f \text{ $M$-a.e.\ for some  } f\in H^1_e(D) \big\}, \\
   \hat{ \mathcal{E}}(\varphi,\varphi)  \ldef \mathcal{E}(H_{\cX} f,H_{\cX} f), \qquad \varphi\in  \hat{\mathcal{F}}, \,  \varphi= f \text{ $M$-a.e., } f\in H^1_e(D).
 \end{cases}
\end{align}
By \cite[Theorem~6.2.1]{FOT11} $(\hat{ \mathcal{E}},  \hat{\mathcal{F}})$ is the regular Dirichlet form on $L^2({\cX};M)$ associated with the process $\mathcal{B}$. Since $\cX$ has Lebesgue measure zero, it follows from the Beurling-Deny representation formula for $\hat{ \mathcal{E}}$ (see  \cite[Theorem~5.5.9]{CF12}) that $\cB$ has no diffusive part and is therefore a pure jump process.

\section{Random walk approximations} \label{sec:rw}
\subsection{Approximation by a random walk on the leaves}
For any $0<r<t$ we define
\begin{align}
  \cB_{r,t}(s) \ldef B_{F_{r,t}^{-1}(s)}, \qquad s\geq 0,
\end{align}
where $F_{r,t}^{-1}$ denotes the right-continuous inverse of $F_{r,t}$.
The process $\cB_{r,t}$ is taking values in $\{ \gamma(x_j(r)), \, j\leq n(t)\}$ and it may therefore be regarded as a random walk on the leaves of the underlying Galton-Watson tree represented by their values under the embedding $\gamma$.

Let $D([0,\infty), \bbR_+)$ (or $D((0,\infty), \bbR_+)$, $D([0,S], \bbR_+)$) be the the space of $\bbR_+$-valued c\`adl\`ag paths on $[0,\infty)$ (or $(0,\infty)$, $[0,S]$). We denote by $d_{J_1}$ and $d_{M_1}$ the metric  w.r.t.\ Skorohod $J_1$- and $M_1$-topology, respectively. We refer to \cite[Chapter 3]{Wh02} for the precise definitions. Further, let
\begin{align}
 D_\uparrow ([0,\infty), \bbR_+) \ldef \big\{ w \in D([0,\infty), \bbR_+): \text{$w$ non-decreasing}, \, w(0)=0  \big\}.
 \end{align}
 Finally, we set
\begin{align}
L^1_{\mathrm{loc}} \ldef \Big\{ w\in D((0,\infty), \bbR_+): \int_0^S |w(s)| \, ds <\infty \quad \text{for all $S\geq 0$}  \Big\},
\end{align}
equipped with the topology induced by supposing 
\begin{align}
w_n \rightarrow w \quad \text{if and only if} \quad \int_0^S |w_n(s)-w(s)| \, ds \rightarrow 0 \quad \text{ for all $S\geq 0$.}
\end{align}
Note that the $L^1_{\mathrm{loc}}$-topology extends both the $J_1$- and the $M_1$-topology since it allows excursions in the approximating processes which are not present in the limit process provided they are of negligible $L^1$-magnitude (cf.\ \cite[Remark~1.3]{CM15}).

\begin{theorem} \label{thm:approx_proc}
$\prob$-a.s., for every $x\in \bbR_+$ we have  under $P_x$,
\begin{align}
 \lim_{r\uparrow \infty} \lim_{t \uparrow \infty} \cB_{r,t} = \cB
\end{align}
in distribution on $L^1_{\mathrm{loc}}$, that is, $\prob$-a.s., for every $x\in \bbR_+$ and 
for all bounded continuous functions $f$ on $L^1_{\mathrm{loc}}$,
\begin{align} \label{eq:convBrt}
  \lim_{r\uparrow \infty} \lim_{t \uparrow \infty}  E_x[f(\cB_{r,t})]  =  E_x[f(\cB)]. 
\end{align}
\end{theorem}

\begin{remark}
Since the measures $M_{r,t}$ and $M$ do not have full support and $F^{-1}_{r,t}$ and $F^{-1}$ have discontinuities, the locally uniform convergence of the functionals $F_{r,t}$ only implies the $M_1$-convergence of their inverses. In such a situation the composition mapping is only continuous in the $L^1_{\mathrm{loc}}$-topology (see Lemma~\ref{lem:cadlag_paths} below), which is why we  obtain the approximation in Theorem~\ref{thm:approx_proc} in the coarser $L^1_{\mathrm{loc}}$-topology only. We refer to \cite[Corollary~1.5 (b)]{CHK16} for a similar result and to \cite{CM15, FM14, MM15} for examples of convergence results for trap models in the $L^1_{\mathrm{loc}}$-topology (or slight modifications of it).   
\end{remark}

Before we prove Theorem~\ref{thm:approx_proc} we recall some facts about the continuity of the inverse and the composition mapping on the space of c\`adl\`ag paths. 

\begin{lemma} \label{lem:cadlag_paths}
 \begin{enumerate}
  \item [(i)] For any $w_1, w_2 \in D([0,S], \bbR_+)$,
  \begin{align} \label{eq:comp_metric}
   d_{M_1}(w_1,w_2) \leq  d_{J_1}(w_1,w_2) \leq \sup_{s\in [0,S]} |w_1(s) - w_2(s)|.
  \end{align}
  \item[(ii)] Let $(a_n)$ be a sequence in  $D_\uparrow([0,\infty), \bbR_+)$ such that $a_n \rightarrow a$ in $M_1$-topology for some $a\in  D_\uparrow([0,\infty), \bbR_+)$.  Then, $a_n^{-1} \rightarrow a^{-1}$  in $D((0,\infty),\bbR_+)$ equipped with $M_1$-topology, where $a_n^{-1}$ and $a^{-1}$ denote the right-continuous inverses of $a_n$ and $a$, respectively. 
 \item[(iii)]   Let $(a_n) \subset D_\uparrow([0,\infty), \bbR_+)$ and $(w_n) \subset D([0,\infty), \bbR_+)$ such that $a_n \rightarrow a$ in $M_1$-topology for some $a\in  D_\uparrow([0,\infty), \bbR_+)$ and $w_n\rightarrow w$ in $J_1$-topology for some $w \in C([0,\infty), \bbR_+)$. Then, $w_n\circ a_n \rightarrow w \circ a$ in $L^1_{\mathrm{loc}}$-topology.
 \end{enumerate}
\end{lemma}
\begin{proof}
 For the first inequality in \eqref{eq:comp_metric} we refer to \cite[Theorem 12.3.2]{Wh02} and the second inequality is immediate from the definition of the $J_1$-metric. Statement (ii) follows from the continuity of the inverse mapping in $D((0,\infty), \bbR_+)$, see \cite[Corollary 13.6.5]{Wh02}. 
For (iii) see \cite[Lemma~A.6]{CM15}.
\end{proof}

\begin{proof}[Proof of Theorem~\ref{thm:approx_proc}]
Fix an environment $\om \in \Omega$ such that Theorem~\ref{thm:constrF} holds giving that for any $x\in \bbR_+$, $P_x$-a.s., $F_{r,t} \rightarrow F$ locally uniformly  as first $t\uparrow \infty$ and then $r\uparrow \infty$. In particular, using Lemma~\ref{lem:cadlag_paths} (i) we have that $F_{r,t} \rightarrow F$ in $M_1$-topology $P_x$-a.s. In particular, for all bounded $\varphi$ acting on $D([0,\infty), \bbR_+)$ which are continuous in $M_1$-topology on a set with full $P_x$-measure,
\begin{align} \label{eq:weak_convF}
  \lim_{r\uparrow \infty} \lim_{t \uparrow \infty} E_x\big[\varphi(F_{r,t})\big]  = E_x\big[\varphi(F)\big]. 
\end{align}
Now, observe that for any bounded continuous $f$ on  $L^1_{\mathrm{loc}}$,
 \begin{align}
 E_x\big[f(\cB_{r,t}) - f(\cB)\big]= E_x\big[f\circ \pi (F_{r,t}, B) - f\circ \pi (F,B)\big],
 \end{align}
 where 
\begin{align}
\pi: \big(D_\uparrow([0,\infty), \bbR_+), d_{M_1} \big) \times 
\big(D([0,\infty), \bbR_+), d_{J_1} \big) \rightarrow L^1_{\mathrm{loc}} \quad
 (a,w) \mapsto w\circ a^{-1}.
\end{align} 
 Thus  Lemma~\ref{lem:cadlag_paths} (ii) and (iii) ensure  the continuity of the mapping $\pi$ in $M_1$-topology on a set with full $P_x$-measure. Hence, \eqref{eq:convBrt} follows from \eqref{eq:weak_convF}.
\end{proof}

\begin{remark}
In the special case $x=0$  the  convergence result in Theorem~\ref{thm:approx_proc} can be extended to $D([0,\infty),\bbR_+)$ equipped with $L^1_{\mathrm{loc}}$-topology. This is because the continuity of the inverse map stated in Lemma~\ref{lem:cadlag_paths}(ii) also holds in $D([0,\infty),\bbR_+)$ under the additional assumption that $a^{-1}(0)=0$ (cf.\ \cite[Chapter 13.6]{Wh02}). 
Note that by construction the origin is contained in $\cX$ so that $F^{-1}(0)=0$ under $P_0$. However, an arbitrary $x>0$ might not be contained in the support $\cX$ of the random measure $M$, in which case $F^{-1}(0)=0$ does not hold. 
\end{remark}

\subsection{Approximation by  random walks on a lattice}

Next we provide approximation results for $\mathcal{B}$ in terms of  random walks on the lattice $(\frac 1 r \bbZ_+)$, $r>0$.
For any $0<r<t$ let $\tilde M_{r,t}$ be the random measure
\begin{align}
\tilde M_{r,t} \ldef \sum_{k=0}^\infty \delta_{\frac  k r} \sum_{j\leq n(t)} \big(\sqrt{2} t -x_j(t)\big) \, e^{\sqrt{2} (x_j(t)-\sqrt{2}t)} \, \indicator_{\big\{ \gamma(x_j(t) \in \left[\frac k r, \frac{k+1}{r}\right)\big\}}
\end{align}
with the associated PCAF $\tilde F_{r,t}: [0,\infty)\rightarrow [0,\infty)$ given by
\begin{align}
\tilde F_{r,t}(s) &\ldef \int_{\bbR_+} L_s^a \, \tilde M_{r,t}(da) \nonumber \\
& =  \sum_{k=0}^\infty L_s^{\frac k r} \,  \sum_{j\leq n(t)} \big(\sqrt{2} t -x_j(t)\big) \, e^{\sqrt{2} (x_j(t)-\sqrt{2}t)} \, \indicator_{\big\{ \gamma(x_j(t)) \in \left[\frac k r, \frac{k+1}{r}\right)\big\}}.
\end{align}
 Then $\tilde \cB_{r,t}(s):= B_{\tilde F_{r,t}^{-1}(s)}$, $s\geq 0$, defines a random walk on $(\frac 1 r \bbZ_+)$. Further, let $P_0^{\mathrm{rw}}$ be the probability measure on $D([0,\infty), \mathbb{R}_+)$, under which the coordinate process  $(X_s)_{s\geq 0}$ is a simple random walk on $\mathbb{Z}_+$ in continuous time with independent $\exp(1)$-distributed holding times. Define $X_{r,t}(s)\ldef \frac 1 r  X_{r^2 \tilde F_{r,t}^{-1}(s)}$, $s\geq 0$.

 \begin{theorem} \label{thm:rw_approx}
(i) For every $x\in \bbR_+$, under the annealed law $\int P_x() \, d\prob$,
\begin{align}
 \lim_{r\uparrow \infty} \lim_{t \uparrow \infty} \tilde\cB_{r,t} = \cB, \qquad \text{in distribution on $L^1_{\mathrm{loc}}$,}
 \end{align}
  that is for all bounded continuous functions $f$ on $L^1_{\mathrm{loc}}$ we have
\begin{align} \label{eq:convBrt2}
  \lim_{r\uparrow \infty} \lim_{t \uparrow \infty} \bbE\big[ E_x[f(\tilde \cB_{r,t})] \big] = \bbE\big[ E_x[f(\cB)] \big]. 
\end{align}

(ii) Under $\int P_0^{\mathrm{rw}}() \, d\prob$, $ \lim_{r\uparrow \infty} \lim_{t \uparrow \infty} X_{r,t} = \cB$
in distribution on $L^1_{\mathrm{loc}}$, that is for all bounded continuous functions $f$ on $L^1_{\mathrm{loc}}$,
\begin{align} 
  \lim_{r\uparrow \infty} \lim_{t \uparrow \infty}  E_0^{\mathrm{rw}}\big[f(X_{r,t})\big]  =  E_0\big[f(\cB)\big]. 
\end{align}
 \end{theorem}
\begin{remark}
The proof of Theorem~\ref{thm:rw_approx} relies on the locally uniform convergence of $\tilde F_{r,t}$ towards $F$ in $\prob \times P_x$-probability, see Proposition~\ref{prop:tildeF} below. Similarly, by using Theorem~\ref{thm:constrF} instead, one can show that  $\prob$-a.s., under  $P_0^{\mathrm{rw}}$, the processes
$\big( \frac 1 r  X_{r^2 F_{r,t}^{-1}(s)}\big)_{s\geq 0}$ converge towards $\cB$
in distribution on $L^1_{\mathrm{loc}}$.
\end{remark} 
The proof of Theorem~\ref{thm:rw_approx} requires some preparations.  For $0\leq r<t<\infty$ set
\begin{align}
 Z_{r,t}^\gamma \ldef \sum_{k=1}^{n(t)} \big( \sqrt{2} t -x_k(t) \big)\,  e^{\sqrt{2}(x_k(t)-\sqrt{2}t)} \, \indicator_{\Delta_{r,t}^k},
\end{align}
where $\Delta^k_{r,t}\ldef \{\vert\gamma(x_k(t)-\gamma(x_k(r))\vert \leq e^{-r/2}\}$.
Next we show that this thinned $Z_{r,t}^\gamma $, which only keeps track of particles whose values under $\gamma$ do not change much over time, is close to the original measure $Z_{r,t}$ in probability.
\begin{lemma} \label{lem:compZ}
For any $\varepsilon, \delta>0$ there exist $r_0=r_0(\varepsilon)$ and $t_0=t_0(\varepsilon)$ such that for any $r>r_0$ and $t>3r \vee t_0$,
 \begin{align}
 \bbP\big[\vert Z_t-Z_{r,t}^{\gamma}\vert>\delta\big] <\varepsilon.
 \end{align}
\end{lemma}
\begin{proof}
 For $d\in\bbR$ and $0\leq r<t\leq u<\infty$ we define the event
 \begin{align} \label{eq:defB}
  \cA_{r,t,u}(d) & \ldef \big\{ \forall k  \leq n(u) \text{ with } x_k(u)-m(u) > d: |\gamma(x_k(t))-\gamma(x_k(r)) | \leq e^{-r/2} \big\}.
 \end{align}
Let $\mathcal{F}_t\ldef \sigma\big\{(x_k(s))_{1\leq k \leq n(s)}, s\leq t \big\}$ and for  $\underline{A}, \overline{A}\in \mathbb{R}$ with $\underline{A}< \overline{A}$ we set $\phi(x)\ldef\indicator_{[\underline{A},\overline{A}]}(x)$. We observe that for any $t>0$ the martingale $Z_t$ appeared  in \cite{ABK_P} (see Eq.\ (3.25) therein) in the $\prob$-a.s.\ limit of
\begin{align}\label{eq:ulimit0}
& \lim_{u\uparrow \infty}\bbE\Bigg[\bbE\bigg[\exp\!\Big(-\sum_{i=1}^{n(u)}\phi\big(x_i(u)-m(u)\big) \Big) \, \Big\vert \mathcal{F}_t\bigg]\Bigg] \nonumber \\
& \mspace{36mu} = c_t \, \bbE\bigg[\exp\!\Big(-C  \big(e^{-\sqrt{2}\underline{A}}-e^{-\sqrt{2}\,\overline{A}}\big) Z_{t}\Big)\bigg],
\end{align}
where $\lim_{t\uparrow \infty} c_t=1$ and $C$ is the same constant as in \eqref{extremal.1.1}. Similarly, for any $0<r<t$ we can consider
\begin{eqnarray}\label{neu.1}
\lim_{u\uparrow\infty}\bbE\bigg[\exp\! \Big( -\sum_{i=1}^{n(u)}\indicator_{\Delta^i_{r,t}} \, \phi\big(x_i(u)-m(u)\big)\Big) \bigg]. 
\end{eqnarray}
Note that $\indicator_{\Delta^i_{r,t}}$ is measurable with respect to $\mathcal{F}_t$.  Then, the limit in \eqref{neu.1} can be treated similarly as the one in \cite[Eq.\ (3.17)]{ABK_P}. More precisely, by repeating the analysis therein (where the sum in the analogue to \cite[Eq.\ (3.19)]{ABK_P} runs over particles with $|\gamma(x_i(t))-\gamma(x_i(r)) | \leq e^{-r/2}$ only)  we obtain
\begin{align} \label{eq:ulimit}
&\lim_{u\uparrow \infty}\bbE\Bigg[ \bbE\bigg[ \exp\! \Big(-\sum_{i=1}^{n(u)}\indicator_{\Delta^i_{r,t}} \phi\big(x_i(u)-m(u)\big)\Big) \, \Big|  \mathcal{F}_t \bigg] \Bigg] \nonumber\\
& \mspace{36mu} 
=  \, c'_t \, \bbE\left[ \exp\!\Big(-C  \big(e^{-\sqrt{2}\underline{A}}-e^{-\sqrt{2}\,\overline{A}}\big) Z_{r,t}^\gamma\Big)\right],
\end{align}
where $\lim_{t\uparrow \infty} c'_t=1$. Moreover, the expectations in \eqref{eq:ulimit0} and \eqref{eq:ulimit}  can be related as follows,
\begin{align}\label{Z.200}
&  \bbE\bigg[\exp\!\Big(-\sum_{i=1}^{n(u)}\indicator_{\Delta^i_{r,t}} \phi\big(x_i(u)-m(u)\big)\Big)  \bigg]\nonumber\\
\geq& \,  \bbE\bigg[\exp\!\Big(-\sum_{i=1}^{n(u)}\phi(x_i(u)-m(u)) \Big)  \bigg]
\geq \,  \bbE \bigg[\exp\! \Big( -\sum_{i=1}^{n(u)} \phi\big(x_i(u)-m(u)\big)\Big) \indicator_{ \cA_{r,t,u}(\underline{A})}  \bigg] \nonumber\\
= & \, \bbE\bigg[\exp\!\Big(-\sum_{i=1}^{n(u)}\indicator_{\Delta^i_{r,t}} \phi\big(x_i(u)-m(u)\big)\Big) \indicator_{ \cA_{r,t,u}(\underline{A})}  \bigg]\nonumber\\
\geq&  \, \bbE\bigg[\exp\!\Big(-\sum_{i=1}^{n(u)} \indicator_{\Delta^i_{r,t}}\phi\big(x_i(u)-m(u)\big)\Big)   \bigg]-\bbP\Big[(\cA_{r,t,u} (\underline{A}))^c\Big].
\end{align}

Let $\varepsilon>0$. By \cite[Lemma~4.2]{BH14} there  exist $r_0(\varepsilon)$ and  $t_0(\varepsilon)$ such that for all $t\geq t_0(\varepsilon)$ and $r>r_0(\varepsilon)$, 
\begin{align}\label{Z.100}
\lim_{u\uparrow \infty}\bbP\Big[(\cA_{r,t,u} (\underline{A}))^c\Big]<\varepsilon.
\end{align}
Hence, by combining \eqref{Z.200} with \eqref{eq:ulimit0} and \eqref{eq:ulimit} we get 
\begin{align}\label{Z.201}
&c'_t \, \bbE\!\left[\exp\!\Big(-C  \big(e^{-\sqrt{2}\underline{A}}-e^{-\sqrt{2}\,\overline{A}}\big) Z_{r,t}^\gamma\Big)\right]- \varepsilon\nonumber\\
\leq & \,  c_t \, \bbE\!\left[\exp \!\Big(-C  \big(e^{-\sqrt{2}\underline{A}}-e^{-\sqrt{2}\,\overline{A}}\big) Z_{t}\Big)\right]\nonumber\\
\leq & \,  c'_t \, \bbE\!\left[\exp\!\Big(-C \big(e^{-\sqrt{2}\underline{A}}-e^{-\sqrt{2}\,\overline{A}}\big) Z_{r,t}^\gamma \Big)\right].
\end{align}
Recall that   $Z_t\rightarrow Z$ $\bbP$-a.s.\ as $t\to \infty$ (cf.\ \cite{LS87}), where $Z$ is $\bbP$-a.s.\ positive,
 and $\lim_{t\uparrow \infty} c_t=\lim_{t\uparrow\infty} c_t'=1$. Hence, for all $t$ and $r$ sufficiently large,
 \begin{equation}
\bbP\!\left[\left\vert \exp\!\Big(-C  \big(e^{-\sqrt{2}\underline{A}}-e^{-\sqrt{2}\,\overline{A}}\big) Z_{r,t}^\gamma \Big) - \exp \!\Big(-C  \big(e^{-\sqrt{2}\underline{A}}-e^{-\sqrt{2}\,\overline{A}}\big) Z_{t}\Big)\right\vert > \delta \right]<\varepsilon.
\end{equation} 
The claim now follows from the continuous mapping theorem since $\exp$ is injective and continuous.
\end{proof}

In the next lemma we lift the statement of Lemma~\ref{lem:compZ} on the level of the PCAFs, meaning that with high probability the PCAFs $F_{r,t}$ and $\tilde F_{r,t}$ are close to their thinned versions $F_{r,t}^\gamma$ and $\tilde F_{r,t}^\gamma$ defined by
\begin{align}
 F_{r,t}^\gamma(s) &  \ldef  \sum_{j=1}^{n(t)} \big( \sqrt{2} t -x_j(t) \big) e^{\sqrt{2}(x_j(t)-\sqrt{2}t)} \indicator_{\Delta^j_{r,t}} \, L^{\gamma(x_j(r))}_s, \\
\tilde F_{r,t}^\gamma(s) &\ldef   \sum_{k=0}^\infty L_s^{\frac k r} \,  \sum_{j\leq n(t)} \big(\sqrt{2} t -x_j(t)\big) \, e^{\sqrt{2} (x_j(t)-\sqrt{2}t)}\, \indicator_{\Delta^j_{r,t}} \, \indicator_{\big\{ \gamma(x_j(t)) \in \left[\frac k r, \frac{k+1}{r}\right)\big\}}. 
\end{align}

\begin{lemma} \label{lem:compF}
 For any $\varepsilon, \delta > 0$ and any $S>0$ there exist $r_1=r_1(\varepsilon, \delta,S)$ and  $t_1=t_1(\varepsilon, \delta, S)$ such that for all $r>r_1$ and  $t>3r \vee t_1$ the following holds.
 There exists a set $\Lambda_1=\Lambda_1(\varepsilon, \delta, S, r,t) \subset \Omega \times \Omega'$ with $\prob_x[\Lambda_1^c]<\varepsilon$ for all  $x\in \bbR_+$ such that on $\Lambda_1$,
\begin{align}
 \sup_{s\leq S}\left\vert F_{r,t}(s)-F_{r,t}^\gamma(s)\right\vert \leq \delta, \qquad  \sup_{s\leq S}\left\vert \tilde F_{r,t}(s)-\tilde F_{r,t}^\gamma(s)\right\vert \leq \delta.
\end{align}
\end{lemma}
\begin{proof}
Recall that by Lemma~\ref{lem:Z} for $\prob$-a.e.\ $\om$ there exists $\tau_0=\tau_0(\om)$ such that 
 $\min_{i\leq n(t)}\left(\sqrt 2 t-x_i(t)\right) >0$ for all $t>\tau_0$. 
 Further, Lemma~\ref{lem:tail_L} gives that for any $\varepsilon >0$ there exists $\lambda=\lambda(\varepsilon, S)$ such that for all $x\in \bbR_+$,
 \begin{equation}\Eq(Lisa.5)
 P_x\Big[ \sup_{a\in \bbR_+} L^{a}_S> \lambda \Big]<\varepsilon.
 \end{equation} 
Together with Lemma~\ref{lem:compZ} this implies that there exist $r_1=r_1(\varepsilon, \delta, S)$ and $t_1=t_1(\varepsilon, \delta, S)$ such that for all $r>r_1$ and $t>3r \vee t_1$ there is a set $\Lambda_1=\Lambda_1(\varepsilon, \delta, S, r,t)$ with $\prob_x[\Lambda^c]<\varepsilon$ for all  $x\in \bbR_+$ on which
\begin{itemize}
 \item $t>\tau_0$,
 \item $\sup_{a\in \bbR_+} L^{a}_S \leq \lambda$,
 \item $\vert Z_t-Z_{r,t}^\gamma\vert \leq \delta/\lambda$.
\end{itemize}
Note that on the set $\Lambda_1$,
\begin{align}
\sup_{s\leq S}\left\vert F_{r,t}(s)-F_{r,t}^\gamma(s)\right\vert & \leq  \vert Z_t-Z_{r,t}^\gamma\vert \, \sup_{s\leq S}\max_{k\leq n(t)} L^{\gamma(x_k(r))}_s  \nonumber \\
& \leq  \vert Z_t-Z_{r,t}^\gamma\vert \,  \sup_{a\in \bbR_+} L^{a}_S \leq\delta,
\end{align}
which completes the proof  of the first statement. The second statement can be shown by similar arguments.
\end{proof}
 In the following we will write $\bbP_x\ldef \bbP\! \times \! P_x$,  $x\in \bbR_+$ for abbreviation. 
 
 \begin{prop} \label{prop:tildeF}
 For every $x\in \bbR_+$ and any $S>0$,
 \begin{align}
 \lim_{r\uparrow \infty}  \lim_{t\uparrow \infty}  \sup_{s\leq S}\big| \tilde F_{r,t}(s) -F(s)\big|=0, \qquad \text{in $\bbP_x$-probability}.
 \end{align}
 \end{prop}
 
\begin{proof}
In view of  Theorem~\ref{thm:constrF}(i)  and Lemma~\ref{lem:compF} it suffices to show that
 \begin{align} \label{eq:conv_thinnedF}
 \lim_{r\uparrow \infty}  \lim_{t\uparrow \infty}  \sup_{s\leq S}\big| \tilde F^\gamma_{r,t}(s) -F_{r,t}^\gamma(s)\big|=0, \qquad \text{in $\bbP_x$-probability}.
 \end{align}
  By Lemma~\ref{lem:Z}, $\prob$-a.s., there exists $\tau_0$ such that $\min_{j\leq n(t)} (\sqrt{2}t -x_i(t))>0$ for all $t\geq \tau_0$, and for such $t$ and  any $x\in\bbR_+$ we get 
 \begin{align}
& \sup_{s\leq S} \big| \tilde F^\gamma_{r,t}(s) - F^\gamma_{r,t}(s) \big| \nonumber \\
  \leq &
  \sum_{k=0}^\infty   \sum_{j\leq n(t)} \big(\sqrt{2} t -x_j(t)\big) \, e^{\sqrt{2} (x_j(t)-\sqrt{2}t)} \,
 \sup_{s\leq S} \Big| L_s^{\frac k r}  - L_s^{\gamma(x_j(r))} \Big| 
   \indicator_{\Delta_{r,t}^j \cap \left\{ \gamma(x_j(t)) \in \left[\frac k r, \frac{k+1}{r}\right)\right\}}. 
\end{align}
Note that  on the event $\Delta_{r,t}^j \cap \{ \gamma(x_j(t)) \in \left[\frac k r, \frac{k+1}{r}\right)\}$ we have
\begin{align}
\gamma(x_j(r)) \in \big[\tfrac k r -e^{-r/2}, \tfrac{k+1}{r}+e^{-r/2}\big),
\end{align} 
which implies  $\big| \frac k r - \gamma(x_j(r))\big| \leq \frac 1 r + e^{-r/2}$. Hence, by Lemma~\ref{lem:propL}(ii), $\prob_x$-a.s.,
\begin{align}
& \sup_{s\leq S} \big| \tilde F^\gamma_{r,t}(s) - F^\gamma_{r,t}(s) \big|  \leq 
C_1 \, \big(\tfrac 1r +e^{-r/2}\big)^{\alpha} \, Z_t.
\end{align}
Recall that $\prob$-a.s.\ $Z_t\rightarrow Z$ as $t\to \infty$  again by Lemma~\ref{lem:Z}, and we obtain  \eqref{eq:conv_thinnedF}.
\end{proof}

\begin{proof}[Proof of Theorem~\ref{thm:rw_approx}]
(i) By Proposition~\ref{prop:tildeF}, $\tilde F_{r,t} \rightarrow F$ locally uniformly in $\prob_x$-probability as first $t\uparrow \infty$ and then $r\uparrow \infty$. In particular, using Lemma~\ref{lem:cadlag_paths} (i) we have that $\tilde F_{r,t} \rightarrow F$ in $M_1$-topology in $\prob_x$-distribution, that is for all bounded $\varphi$ acting on $D((0,\infty), \bbR_+)$ which are continuous in $M_1$-topology on a set with full $\bbP_x$-measure,
\begin{align} 
  \lim_{r\uparrow \infty} \lim_{t \uparrow \infty} \bbE\big[ E_x[\varphi(F_{r,t})] \big] = \bbE\big[ E_x[\varphi(F)] \big]. 
\end{align}
The claim follows now similarly as in the proof of Theorem~\ref{thm:approx_proc} above.

(ii) Recall that $(\frac 1 r X_{r^2 s})_{s\geq 0}$ converges towards $B\in C([0,\infty),\bbR_+)$ in distribution on $D([0,\infty), \mathbb{R}_+)$ in $J_1$-topology.  The statement now follows from Proposition~\ref{prop:tildeF} and Lemma~\ref{lem:cadlag_paths} similarly as in the proof of (i) and Theorem~\ref{thm:approx_proc} (cf.\ \cite[Corollary~1.5]{CHK16}). 
\end{proof}

\section{The subcritical case} \label{sec:subcrit}
 
Recall that the McKean-martingale is defined as
\begin{equation}
Y^{\sigma}_t\ldef \sum_{i=1}^{n(t)} e^{\sqrt{2}\sigma x_k(t)-(1+\sigma^2)t}, \qquad \sigma\in (0,1),
\end{equation}
which is  normalised to have mean $1$. By \cite[Theorem~4.2]{BH13} the limit
\begin{align}
 Y^\sigma \ldef \lim_{t \uparrow \infty} Y^\sigma_t
\end{align}
exists $\mathbb{P}$-a.s.\ and in $L^1(\mathbb{P})$.
For $v,r\in\mathbb{R}_+$ and $t>r$, we define a truncated version of the McKean-martingale $Y^{\sigma}_t$ by
\begin{equation}\label{Y.1}
Y^{\sigma}_{r,t}(v)\ldef\sum_{j\leq n(t)}e^{\sqrt 2 \sigma x_j(t)-(1+\sigma^2) t}\indicator_{\{\gamma(x_i(r))\leq v\}}.
\end{equation}
 
\begin{prop}\label{lem.Y}
 For each $v\in\mathbb{R}_{+}$ the limit   
 \begin{equation}\label{Z.2}
 Y^{\sigma}(v)\ldef\lim_{r\uparrow \infty}\lim_{t\uparrow \infty}Y_{r,t}^{\sigma}(v) 
 \end{equation}
  exists $\mathbb{P}$-a.s. In particular, $0\leq Y^{\sigma}(v) \leq Y^{\sigma}$. Moreover, $Y^{\sigma}(v)$ is increasing in $v$ and the corresponding Borel measure $M^\sigma$ on $\bbR_+$, defined via $M^\sigma([0,v])=Y^\sigma(v)$ for all $v\in \bbR_+$, is $\prob$-a.s.\ non-atomic.
\end{prop}
\begin{proof}
This follows by the same arguments as in \cite[Proposition~3.2]{BH14}. Observe that $Y_{r,t}^{\sigma}(v)$ is non-negative by definition.
\end{proof}

Our goal is to state an analogue to Theorem~\ref{thm:constrF} for the subcritical case. This will be done in Subsection~\ref{sec:subcr_result} below.
First we notice that in the subcritical case  the martingales $Y^{\sigma}$ with $\sigma<1$ appear in the description of the limiting extremal process of two speed branching Brownian motion and that the extended convergence result can be transferred  to this class of models. This is the purpose of Subsection~\ref{sec:twoSpeed}.

\subsection{The extremal process of two-speed branching Brownian motion} \label{sec:twoSpeed}
Next we recall the characterisation of the extremal process for a two-speed branching Brownian motion established in \cite{BH13}. For a fixed time $u$, a two-speed BBM is defined similarly as the ordinary BBM but at time $t'$ the particles move as independent Brownian motions with variance
\begin{align} \label{speed.1}
\sigma^2(t')= \begin{cases}			
\sigma_1^2,& 0\leq t' < bu,\\
\sigma_2^2, & bu\leq t'\leq u,
\end{cases}
\qquad 0<b\leq 1,
\end{align}   
where the  total variance is normalised by assuming $b\sigma_1^2 + (1-b) \sigma_2^2=1$.
Then,  if $\sigma_1<\sigma_2$  the limit $Y^{\sigma_1}$ of the McKean-martingale  appears in the extremal process of the two-speed BBM. More precisely, we have the following result proven in \cite[Theorem~1.2]{BH13}.

 \begin{theorem} \label{thm:twoSpeed}
 Let $ \tilde x_k(u)$ be a branching Brownian motion with variable speed  $\sigma^2(t')$ as given 
in \eqref{speed.1}.  Assume that $\sigma_1<\sigma_2$. Then, 
\begin{enumerate} 
\item[(i)] $\lim_{u\uparrow \infty} \mathbb{P}\left(\max_{k\leq n(u)} \tilde x_k(u)-\tilde m(u) \leq y\right)
= \mathbb{E} \Big[ \exp\big(-C(\sigma_2) Y^{\sigma_1} e^{-\sqrt 2 y}\big) \Big]$, \newline
where $\tilde m(u)=\sqrt 2 u-\frac{1}{2\sqrt 2} \log u$ and $C(\sigma_2)$ is a constant depending on $\sigma_2$. 

\item [(ii)] The point process 
\begin{equation}
\label{main.2}
\sum_{k\leq n(u)} \delta_{\tilde x_k(u)-\tilde m(u)}\Rightarrow \sum_{i,j}\delta_{\eta_i+\sigma_2\Lambda^{(i)}_j} \qquad \text{as $u\uparrow \infty$  in law.} 
\end{equation}
Here $\eta_i$ denotes the $i$-th atom of a mixture of Poisson point process with intensity 
measure $C(\sigma_2)Y^{\sigma_1} e^{-\sqrt 2 y}dy$ with $C(\sigma_2)$ as in (i),  
and
$\Lambda^{(i)}_j$ are the atoms of independent and identically distributed  point processes $\Lambda^{(i)}$, which are the limits in law
of 
\begin{equation}
\sum_{k\leq n(u) }\delta_{\bar x_k(u)-\max_{j\leq n(u)}\bar x_j(u)},
\label{main.4}
\end{equation}
where $\bar x(u)$ is a BBM of speed $1$ conditioned on $\max_{j\leq n(u)} \bar x_j(u)\geq \sqrt 2\sigma_2 t$.
\end{enumerate}
\end{theorem}

Using the embedding $\gamma$ the convergence result in Theorem~\ref{thm:twoSpeed} can be extended as follows.
\begin{theorem}\label{thm.extend}
The point process 
\begin{align}
 \sum_{k=1}^{n(t)} \delta_{(\gamma(u^k(u)), \tilde x_k(u)-\tilde m(u))}\Rightarrow \sum_{i,j}\delta_{(q_i,p_i)+(0,\Lambda^{(i)}_j)}
\end{align}
in law on $\mathbb{R}_+\times \mathbb{R}$, as
$u\uparrow \infty$, where  $(q_i,p_i)_{i\in \mathbb{N}}$ are the atoms of a Cox process on $\mathbb{R}_+\times \mathbb{R}$ with intensity measure
$M^{\sigma_1}( dv)\times C(\sigma_2)e^{-\sqrt{2} x}dx$, where $M^{\sigma_1}(dv)$ is the random measure on $\mathbb{R}_+$ characterised in Proposition \ref{lem.Y}, and $\Lambda_j^{(i)}$  are the atoms of  independent and identically distributed point processes $\Lambda^{(i)}$ as in Theorem~\ref{thm:twoSpeed} (ii).
\end{theorem}
\begin{proof}
The proof goes along the lines of the proof of \cite[Theorem~3.1]{BH14}. Note that by the localisation of the path of extremal particles given in \cite[Proposition~2.1]{BH13} the thinning can be applied in the same way using \cite[Proposition~3.1]{BH13} which provides the right tail bound on the maximum. 
This gives an alternative way to get the convergence of the local maxima to a Poisson point process. There the McKean-martingale $Y^{\sigma_1}_t$ appears naturally instead of the derivative martingale and one proceeds as in the proof of \cite[Theorem~3.1]{BH14}.
\end{proof}

\subsection{Approximation of the PCAF and the process} \label{sec:subcr_result}
Similarly as in the critical case, for any fixed $\sigma \in (0,1)$ we define the measure $ M^{\sigma}_{r,t}$ on $\bbR_+$ associated with $Y^{\sigma}_{r,t}$ by
\begin{align}
 M^{\sigma}_{r,t} \ldef  \sum_{j\leq n(t)} e^{\sqrt 2 \sigma x_j(t)-(1+\sigma^2) t} \delta_{\gamma(x_j(r))}.
\end{align}
Then Theorem~\ref{thm.extend} implies that $\prob$-a.s.
\begin{align}
 M^{\sigma} = \lim_{r\uparrow \infty} \lim_{t\uparrow \infty} M^{\sigma}_{r,t} \qquad \text{vaguely}.  
\end{align}
Again we are aiming to lift this convergence on the level of the associated PCAFs.

\begin{prop} \label{prop:FrtPCAF.2}
 Let $\sigma\in(0,1)$ be fixed. Then, $\prob$-a.s., for any $0\leq r <t$ the following hold.
\begin{enumerate}
 \item[(i)] The unique PCAF of $B$ with Revuz measure $M_{r,t}^\sigma$ is given by
\begin{align}\Eq(lisa.F2)
 F^{\sigma}_{r,t}: [0,\infty)\rightarrow [0,\infty) \quad  s \mapsto    \sum_{j=1}^{n(t)}  e^{\sqrt{2} \sigma x_j(t)-(1+\sigma^2)t} L^{\gamma(x_j(r))}_s.
\end{align}

\item[(ii)]
There exists a set $\Lambda \subset \Omega'$ with $P_x[\Lambda]=1$ for all $x\in \bbR_+$, on which $F^\sigma_{r,t}$ is continuous, increasing and
satisfies $F^\sigma_{r,t}(0)=0$ and $\lim_{s \to \infty} F^\sigma_{r,t}(s)=\infty$.
\end{enumerate}
\end{prop}
\begin{proof}
This is again a direct consequence from the properties of Brownian local times in Lemma~\ref{lem:loc_dirac} and \ref{lem:propL}.
Note that in this setting the positivity is clear since  $\exp$ is a positive function.
\end{proof}
Next we define
 \begin{align}
 F^\sigma(s):= \int_{\bbR_+} L_s^a \, M^\sigma(da), \qquad s\geq 0.
 \end{align}

\begin{theorem} \label{thm:constrF.2}
 Let $\sigma\in(0,1)$ be fixed. Then  $\prob$-a.s.\ the following hold. 
\begin{enumerate}
 \item[(i)] There exists a set $\Lambda \subset \Omega'$ with $P_x[\Lambda]=1$ for all $x\in \bbR_+$ on which 
\begin{align} 
 F^\sigma=\lim_{r\uparrow \infty} \lim_{t\uparrow\infty} F^\sigma_{r,t}, \qquad \text{in $\sup$-norm on $[0,S]$,}
 \end{align}
 for any $S>0$. In particular, $F^\sigma$ is continuous, 	increasing and satisfies $F^\sigma(0)=0$ and $\lim_{s\to \infty} F^\sigma(s)=\infty$.
 \item[(ii)]  The functional $F^\sigma$ is  the (up to equivalence) unique PCAF of $B$ with Revuz measure $M^\sigma$. 
 \end{enumerate}
\end{theorem}
\begin{proof}
This follows by similar arguments as in the proof of Theorem \ref{thm:constrF} above. 
\end{proof}

Now we define the process $\cB^\sigma(s)\ldef B_{(F^\sigma)^{-1}(s)}$, $s\geq 0$. Similarly as explained in Section~\ref{sec:properties} above for $\cB$, by the general theory of time changes of Markov processes the process $\cB^\sigma$ is a recurrent, $M^\sigma$-symmetric pure jump diffusion on the support of $M^\sigma$ and its Dirichlet form can be abstractly described. For $0<r<t$ let
\begin{align}
  \cB_{r,t}^\sigma(s) \ldef B^\sigma_{F_{r,t}^{-1}(s)}, \qquad s\geq 0.
\end{align}
Then, from Theorem \ref{thm:constrF.2} we obtain as in the critical case the convergence of the associated process.
 
\begin{theorem} \label{thm:approx_proc.2}
$\prob$-a.s., for every $x\in \bbR_+$ we have  under $P_x$,
\begin{align}
 \lim_{r\uparrow \infty} \lim_{t \uparrow \infty} \cB^\sigma_{r,t} = \cB^\sigma
\end{align}
in distribution on $L^1_{\mathrm{loc}}$.
\end{theorem}
\begin{proof}
This can be shown by the same arguments as Theorem \ref{thm:approx_proc}.
\end{proof}
Similarly as discussed for the critical case in Theorem~\ref{thm:rw_approx} above, an approximation of $\cB^\sigma$ in terms of a random walk on a lattice is also possible.

\appendix

\section{Brownian local times}\label{section.lt}
In this section we consider Brownian local times as an example for a PCAF on the Wiener space and recall some of their properties needed in the present paper. 
Let $(\Omega',\cG, (\cG_t)_{t\geq 0}, (P_x)_{x\in \bbR})$ be the Wiener space as introduced in Section~\ref{sec:prelim} with coordinate process $W$, so that $B\ldef |W|$ becomes a reflected Brownian motion on $\bbR_+$ with a field of local times denoted by $\{ L_t^a, t\geq 0, a\in\bbR_+\}$. 
\begin{lemma} \label{lem:propL}
There exists a set $\Lambda \subset \Omega'$ with $P_x[\Lambda]=1$ for all $x\in \bbR_+$ such that for all $\om'\in \Lambda$ the following hold.
\begin{enumerate}
 \item [(i)] For every $a\in \bbR_+$ the mapping $t \mapsto L_t^a$ is continuous, increasing and satisfies $L_0^a(\om')=0$ and $\lim_{t\to \infty} L^a_t(\om')=\infty$. The measure $dL_t^ a(\om')$ is carried by the set $\{ t\geq 0: B_t(\om)=a\}$.
 \item[(ii)] The mapping $(a,t)\mapsto L_t^a(\om')$ is jointly continuous and for every $\alpha<1/2$ and $T>0$ there exists $C_1=C_1(\om',\alpha, T)$ satisfying $\sup_{x\in \bbR_+} E_x[C_1]<\infty$ such that
\begin{align} \label{eq:cont_loctimes}
 \sup_{t\leq T} \big| L_t^a(\om') - L_t^ b(\om') \big|\leq C_1 \, |a-b|^ \alpha.
\end{align}
\end{enumerate}
\end{lemma}
\begin{proof}
 These properties are immediate from \eqref{eq:def_L} since the Brownian local time $L(W)$ satisfies them. We refer to \cite[Chapter VI]{RY99} for details, in particular \cite[Corollary VI.2.4]{RY99} for (i)  and \cite[Theorem VI.1.7 and Corollary VI.1.8]{RY99}) for (ii)  (cf.\  also \cite[Example 5.1.1]{FOT11}).
\end{proof}

\begin{lemma} \label{lem:tail_L}
 For any $t>0$  there exists $\lambda_0=\lambda_0(t)>0$ and a positive constant $C_2$ such that  
 \begin{align} \label{eq:tail_supL}
  P_x\Big[\sup_{a \in \bbR_+} L_t^a > \lambda \Big]\leq C_2 \, \frac{\lambda}{\sqrt{t}} e^{-\lambda^2/2t}, \qquad \forall x\in \bbR_+,\ \lambda \geq \lambda_0.
 \end{align}
 In particular, $\sup_{a \in \bbR_+} L_t^a \in L^2(P_x)$ for any $x\in \bbR_+$.
%
\end{lemma}
\begin{proof}
 In view of \eqref{eq:def_L} it suffices to consider the local times $L_t^a(W)$ of the standard Brownian motion $W$. 
  Note that the event $\big\{ \sup_{a \in \bbR} L_t^a(W) > \lambda \big\}$ does not depend on the starting point of $W$. Under $P_0$ the tail estimate in \eqref{eq:tail_supL} for $\sup_{a \in \bbR} L_t^a(W)$ has been shown in \cite[Lemma~1]{Cs89}.
 The fact that $\sup_{a \in \bbR_+} L_t^a \in L^2(P_x)$ follows from \eqref{eq:tail_supL} by integration.
\end{proof}

Recall that in dimension one only the empty set is polar for $W$ or $B$, so trivially any $\sigma$-finite measure $ \mu$ on $\bbR$ does not charge polar sets and by general theory (see e.g.\ \cite[Theorem 4.1.1]{CF12}) there exist unique (up to equivalence) PCAF $A$ of $W$ or $B$ with $\mu_A=\mu$. In particular, for any $a\in \bbR$  the unique PCAF of $W$ having the Dirac measure $\delta_a$ as Revuz measure is given by $L^a(W)$, see \cite[Example 5.1.1]{FOT11} or \cite[Proposition X.2.4]{RY99}. This can be easily transferred to the reflected Brownian motion.

\begin{lemma} \label{lem:loc_dirac}
 For any $a\in \bbR_+$, the local time $L^ a$ is the PCAF of $B$ with Revuz measure $\delta_a$. 
\end{lemma}
\begin{proof}
 We need to show that for any for any non-negative Borel function $f$ on $\bbR_+$,
 \begin{align} \label{eq:rev_dirac}
  f(a)=\lim_{t\downarrow 0} \frac 1 t \int_{\bbR_+} E_x \Bigl[ \int_0^t f(B_s) \, dL^a_s \Bigr] \, dx.
 \end{align}
We extend $f$ to a function $\tilde f$ on $\bbR$ by setting $\tilde f (x) \ldef f(|x|)$, $x\in \bbR$. Using that $L^a(W)$ is the unique PCAF of $W$ with $\mu_{L^a(W)}=\delta_a$ and that for any $x\in \bbR$ the measure $dL^a(W)$ is $P_x$- a.s.\ carried by the set $\{t: W_t=a\}$  we have
\begin{align} \label{eq:calc_fa}
& f(a)=\tilde f(a)=\lim_{t\downarrow 0} \frac 1 t \int_{\bbR} E_x \Bigl[ \int_0^t \tilde f(W_s) \, dL^a_s(W) \Bigr] \, dx  \\
&=\lim_{t\downarrow 0} \frac 1 t \int_{\bbR_+} \! E_x \Bigl[ \int_0^t f(B_s) \, dL^a_s(W) \Bigr] \, dx + \lim_{t\downarrow 0} \frac 1 t \int_{-\infty}^0 \! E_x \Bigl[ \int_0^t \tilde f(W_s) \, dL^a_s(W) \Bigr] \, dx. \nonumber
\end{align}
Since $L^a(-W)=L^{-a}(W)$ (cf.\ \cite[Exercise VI.1.17]{RY99}) we get
\begin{align}
 \int_{-\infty}^0 \! E_x \Bigl[ \int_0^t \tilde f(W_s) \, dL^a_s(W) \Bigr] \, dx&= \int_{-\infty}^0 \! E_{-x}  \Bigl[ \int_0^t \tilde f(-W_s) \, dL^a_s(-W) \Bigr] \, dx \nonumber\\
 &=\int_{\bbR_+} \! E_x \Bigl[ \int_0^t  f(B_s) \, dL^{-a}_s(W) \Bigr] \, dx
\end{align}
and combining this with \eqref{eq:calc_fa} and \eqref{eq:def_L} we obtain \eqref{eq:rev_dirac}.
\end{proof}
\begin{lemma}\label{lem:int_meanL}
 For any $a\in \bbR_+$, 
 \begin{align}
 \int_{\bbR_+} E_x[L_1^a] \, dx=1.
 \end{align}
\end{lemma}
\begin{proof}
 Recall that
 \begin{align}
  E_x \big[ L_1^a \big]=  E_x \big[ L_1^a(W) \big]+  E_x \big[ L_1^{-a}(W) \big]=  E_a \big[ L_1^x(W) \big] +  E_a \big[ L_1^{-x}(W) \big].
 \end{align}
Hence, by the occupation times formula we obtain
\begin{align}
 \int_{\bbR_+} E_x[L_1^a] \, dx= E_a \Big[ \int_{-\infty}^\infty L_1^x(W) \, dx \Big]=1
\end{align}
(cf.\ \cite[proof of Proposition~X.2.4]{RY99}).
\end{proof}

%

\bibliographystyle{abbrv}
\bibliography{literature}

\end{document}